\documentclass[12pt,leqno]{amsart}
\usepackage{amsmath,amsfonts,amssymb,amsthm}
\usepackage{graphicx}
\graphicspath{ }
\usepackage{wrapfig}
\usepackage{accents}
\usepackage{caption}
\usepackage{subcaption}
\usepackage{calligra}

\numberwithin{figure}{section}

\theoremstyle{plain}
\newtheorem{thm}{Theorem}[section]
\newtheorem{lem}[thm]{Lemma}
\newtheorem{prop}[thm]{Proposition}
\newtheorem{cor}{Corollary}[thm]
\theoremstyle{definition}

\theoremstyle{remark}

\usepackage{mathtools}
\title{Some characterizations of gradient Yamabe solitons}
\author[A. A. Shaikh, P. Mandal]{Absos Ali Shaikh$^{*1}$, Prosenjit Mandal$^{2}$}
\address{\noindent\newline  Department of Mathematics,\newline University of
Burdwan, Golapbag,\newline Burdwan-713104,\newline West Bengal, India}
\email{$^1$aask2003@yahoo.co.in, aashaikh@math.buruniv.ac.in}
\email{$^2$prosenjitmandal235@gmail.com}

\begin{document}
\begin{abstract}
In this article we have proved that a gradient Yamabe soliton satisfying some additional conditions must be of constant scalar curvature. Later, we have showed that in a gradient expanding or steady Yamabe soliton with non-negative Ricci curvature if the potential function satisfies some integral condition then it is subharmonic, in particular, for steady case the potential function becomes harmonic. Also we have proved that, in  a compact gradient Yamabe soliton, the  potential function agrees with the Hodge-de Rham potential upto a constant.
\end{abstract}
\noindent\footnotetext{
$\mathbf{2020}$\hspace{5pt}Mathematics\; Subject\; Classification: 53C20; 53C21. }
\maketitle
\section*{Introduction}
Yamabe solitons are special solutions to the Yamabe flow which is introduced by R. S. Hamilton \cite{RS1988} in the late 20th century. Let $(M,g)$ be a Riemannian manifold of dimension n, $(n \geq 2)$, such that $\{g(t)\}$ is the 1-parameter family of metrics and $R(t)$ is its scalar curvature. Then the equation of Yamabe flow is given by
\begin{equation*}
\frac{\partial g(t)}{\partial t} = -R(t)g(t).
\end{equation*}
A Riemannian manifold $(M,g,\lambda)$ of dimension n, ($n\geq 2$), which is connected, is called a Yamabe soliton if it satisfies
\begin{equation}\label{r10}
\frac{1}{2}\pounds_Xg = (R-\lambda)g,
\end{equation}
where $\pounds_Xg$ is the Lie derivative of the metric g with respect to the smooth vector field X, $\lambda$ is a constant and $R$ denotes the scalar curvature of $M$. A Yamabe soliton  $(M,g,\lambda)$ is called shrinking, steady,  and expanding Yamabe soliton if $\lambda >0$, $\lambda=0,$ and $\lambda <0$, respectively.
If there exists a smooth function $f : M \rightarrow \mathbb{R}$ such that $X= \nabla f$, then the Yamabe soliton is called a gradient Yamabe soliton. We write a gradient Yamabe soliton by $(M, g, f, \lambda)$, and in this case, $(\ref{r10})$ takes the form
\begin{equation}\label{r11}
\nabla^2 f = (R-\lambda)g, 
\end{equation}
where $\nabla^2 f$ denotes the Hessian of the potential function $f$. Moreover, when the vector field X is trivial or the potential function $f$ is constant, then Yamabe soliton will be called trivial, otherwise it will be a non-trivial Yamabe soliton. Taking the trace of equation $(\ref{r11})$ we get
\begin{equation}\label{r12}
\Delta f = (R-\lambda)n.
\end{equation}
In tensors of local coordinate system, the equation $(\ref{r11})$ can be written as
\begin{equation}\label{r13}
\nabla_i \nabla_j f = (R- \lambda)g_{ij}.
\end{equation}  
\indent For a given vector field X on a compact oriented Riemannian manifold $M$, the Hodge-de Rham decomposition theorem, (see e.g. \cite{WF1983}) states that we may decompose X as the sum of a divergence free vector field and gradient of a function h. Hence we set
\begin{equation}
X=Y+\Delta h,
\end{equation}
where divY=0. Here the function $h$ is known as Hodge-de Rham potential. Recently many authors have been studied Yamabe soliton, (see \cite{GA2014, HXY2012, PN2013, SYH2012, LV2012, SM2019} and also the references therein). It is proved, in \cite{PN2013}, that a compact gradient Yamabe soliton $(M, g, f, \lambda )$ is of constant scalar curvature. Also in \cite{SYH2012}, Shu-Yu Hsu provides an alternative proof of this result. In the present paper we omit the compactness condition and with the help of some other conditions, we have proved that the scalar curvature is constant.
Zhu-Hong Zhang \cite{ZHZ2009} proved a result depending on scalar curvature and potential function for gradient Ricci solitons. Following the way of Zhu-Hong Zhang \cite{ZHZ2009}, in this paper we have proved a result depending on scalar curvature and potential function for gradient Yamabe soliton.\newline
\indent At first we have proved that in a gradient Yamabe soliton with potential function bounded below and the manifold satisfying linear volume growth has constant scalar curvature. In the last section we have showed that the potential function in a gradient expanding or steady Yamabe soliton satisfying some integral condition is subharmonic. 
\section{Gradient Yamabe soliton with constant scalar curvature }
\begin{lem}\label{r17}\cite{SC10}
Let $f$ be a non-negative subharmonic function in $B(q,2r)$ then the following inequality holds
 \begin{equation}\label{r16}
 \int_{B(q,r)}|\nabla f|^2 \leq\frac{C}{r^2}\int_{B(q,2r)}f^2,
 \end{equation}
 where $B(q,r)$ is a ball with radius $r>0$ and center at $q$ and C is a real constant.
\end{lem}

\begin{thm}\label{r5}
Let $(M,g,f,\lambda)$ be a gradient Yamabe soliton of dimension n, with the potential function $f\geq K $ for some constant $K>0$, R$\leq\lambda$ and $M$ be of linear volume growth. Then $M$ must be of constant scalar curvature.
\end{thm} 
\begin{proof}
Since $R\leq \lambda$, it follows from equation (\ref{r12}) that $\Delta f $ $\leq 0$. Now
\begin{equation*}
 \Big(\frac{1}{f}\Big)_j = -\Big(\frac{1}{f^2}\Big)f_j,
\end{equation*}

also
\begin{equation*}
 \Big(\frac{1}{f}\Big)_{jj} = \Big(\frac{2}{f^3}\Big)f^2_j -\Big(\frac{1}{f^2}\Big)f_{jj}.
\end{equation*}
Hence
\begin{equation}
\Delta \Big(\frac{1}{f}\Big)=\Big(\frac{2}{f^3}\Big) |\nabla f|^2 -\frac{\Delta f}{f^2}.
\end{equation}
As $\Delta f\leq 0$, it follows that $\Delta (\frac{1}{f})\geq 0$. Since the manifold is of linear volume growth, so the volume of $B(q,r)$, i.e., $V(B(q,r))\leq C_1r$, for some constant $C_1>0$.
Therefore from the equation $(\ref{r16})$, we obtain
\begin{eqnarray*}
\int_{B(q,r)}|\nabla \frac{1}{f}|^2&&\leq \Big(\frac{C}{r^2}\Big)\int_{B(q,2r)}\Big(\frac{1}{f^2}\Big)\\
 && \leq \Big(\frac{C}{r^2K^2}\Big) V(B(q,2r)) \\
 &&\leq \Big(\frac{C}{r^2K^2}\Big)C_1 2r\\
&&\leq \frac{2CC_1}{rK^2} \rightarrow 0  
\end{eqnarray*}
as $r\rightarrow \infty$. Therefore 
\begin{equation}
 \int_{M}|\nabla \frac{1}{f}|^2=0,
\end{equation}
which follows that the function $\frac{1}{f}$ is constant. Consequently the potential function $f$ is constant. Then it follows from (\ref{r12}) that $R= \lambda$, which proves the result.
\end{proof}
\begin{prop}\label{r1}
Let $(M,g,f,\lambda)$ be a gradient Yamabe soliton of dimension n. Then the following relation holds;
 \begin{eqnarray}\label{r2}
 \Delta R =-\frac{1}{n-1}\Big\{   \frac{1}{2} \nabla_lR \nabla_lf + R(R - \lambda \Big \}, 
 \end{eqnarray}
    where R denotes the scalar curvature of $M$. 
\end{prop}
\begin{proof}
Taking the covariant derivative in $(\ref{r13})$ and using the commutating formula for covariant derivative, we have
\begin{equation*}
g_{jk} \nabla_i R - g_{ik} \nabla_j R = \nabla_i\nabla_j\nabla_k f - \nabla_j\nabla_i\nabla_k f = -R_{ijkl}\nabla_lf.
\end{equation*} 
Taking the trace on $j$ and $k$, we get
\begin{equation*}
g_{jj} \nabla_i R - g_{ij} \nabla_j R = - R_{ijjl}\nabla_lf.
\end{equation*}
This implies
\begin{equation*}
n \nabla_i R - \nabla_i R = - R_{il}\nabla_l f,
\end{equation*}
which yields
\begin{equation}\label{r3}
\nabla_i R = - \frac{1}{(n-1)}R_{il}\nabla_l f.
\end{equation}
Now with the help of equation (\ref{r3}) and contracted second Bianchi identity $\nabla_iR =2 \nabla_jR_{ij} $, we get
\begin{eqnarray*}
\Delta R &=& g^{ij}\nabla_i\nabla_jR\\
&=& - g^{ij}\nabla_i\Big(\frac{1}{n-1} R_{jl} \nabla_lf \Big)\\
&=& - \frac{1}{n-1} g^{ij}\nabla_i ( R_{jl} \nabla_lf) \\
&=& - \frac{1}{n-1}g^{ij} \Big \{ \nabla_i(R_{jl}) \nabla_lf + R_{jl} \nabla_i\nabla_lf \Big \}\\
&=& - \frac{1}{n-1} \Big \{ \frac{1}{2} \nabla_lR\nabla_lf + g^{ij}R_{jl}(Rg_{il}-\lambda g_{il})  \Big \}\\
&=& - \frac{1}{n-1} \Big \{  \frac{1}{2} \nabla_lR\nabla_lf +R(R - \lambda)  \Big \}.
\end{eqnarray*}
Hence, we get our result.
\end{proof}

\section{Yamabe soliton and potential functions}
\begin{thm}\label{r9}
Let $(M,g,f,\lambda)$ be a gradient expanding or steady Yamabe soliton of dimension n with non-negative scalar curvature. If the potential function $f$ satisfies the condition
\begin{equation}\label{r6}
 \int_{M-B(q,r)}\frac{f}{d(x,q)^{2}} < \infty,
 \end{equation}
 where $B(q,r)$ is a ball with radius $r>0$ and center at $q$, then $f$ is subharmonic.
\end{thm}
\begin{proof}
Taking trace in (\ref{r11}), we get
\begin{equation}\label{r7}
(R-\lambda)n=\Delta f.
\end{equation}
Let us consider the cut-off function, introduced in \cite{CC1996}, $\psi_r\in C^\infty_0(B(q,2r))$ for $r>0$ such that
\[ \begin{cases} 
	  0\leq \psi_r\leq 1 &\text{ in }B(q,2r)\\
      \psi_r=1  & \text{ in }B(q,r) \\
      |\nabla \psi_r|^2\leq\frac{C}{r^2}& \text{ in }B(q,2r) \\
      \Delta \psi_r\leq \frac{C}{r^2} &  \text{ in }B(q,2r),
   \end{cases}
\]
where $C>0$ is a constant.
Then as $r\rightarrow\infty$, we get $\Delta \psi_r\rightarrow 0$ as $\Delta \psi_r\leq \frac{C}{r^2}$. Now using (\ref{r7}) and integration by parts, we obtain
\begin{eqnarray*}
0\leq \int_{B(q,2r)}\psi_rR = \int_{B(q,2r)} \psi_r (\lambda + \frac{1}{n} \Delta f) &=& \int_{B(q,2r)} \psi_r \lambda + \frac{1}{n} \int_{B(q,2r)} \psi_r \Delta f\\
&\leq& \frac{1}{n} \int_{B(q,2r)-B(q,r)}f \Delta \psi_r\\
&\leq& \frac{1}{n} \int_{B(q,2r)-B(q,r)}\frac{C}{r^2}f \rightarrow 0,
\end{eqnarray*} 
as $r\rightarrow \infty$. Since $\psi_r =1$ in $B(q,r)$, so from above inequality we have $R=0$.
Then from equation (\ref{r12}), we get
\begin{equation*}
\Delta f \geq 0,
\end{equation*}
This completes the proof.
\end{proof}
By the Lemma $\ref{r17}$ and Theorem $\ref{r9}$, we get the following corollary: 
\begin{cor}
Let $(M,g,f,\lambda)$ be a gradient expanding or steady Yamabe soliton of dimension n with non-negative scalar curvature. If a non-negative  potential function $f$ satisfies the condition
\end{cor}
\begin{equation}\label{r23}
  \int_{M-B(q,r)}\frac{f}{d(x,q)^{2}} < \infty,
 \end{equation}
 where $B(q,r)$ is a ball with radius $r>0$ and center at $q$, then $f$ satisfies the inequality
  \begin{equation*}
  \int_{B(q,r)}|\nabla f|^2 \leq\frac{C}{r^2}\int_{B(q,2r)}f^2.
  \end{equation*}
 \begin{cor}
Let $(M,g,f,\lambda)$ be a gradient steady Yamabe soliton of dimension n with non-negative scalar curvature. If the potential function $f$ satisfies
\begin{equation}\label{r25}
 \int_{M-B(q,r)}\frac{f}{d(x,q)^{2}} < \infty,
 \end{equation}
 where $B(q,r)$ is a ball with radius $r>0$ and center at $q$, then $f$ is harmonic.
\begin{proof}
In the above theorem we get the scalar curvature $R=0$, hence for steady Yamabe soliton we have $\Delta f =0$, i.e., $f$ is harmonic. 
\end{proof}
\end{cor}
 \begin{cor}
 Let $(M,g,f,\lambda)$ be a gradient shrinking Yamabe soliton of dimension n with non-negative scalar curvature and $\lambda \leq \Delta \psi$, for some smooth function $\psi$. If the potential function $f$ satisfies
 \begin{equation}\label{r27}
  \int_{M-B(q,r)}\frac{f}{d(x,q)^{2}} < \infty,
  \end{equation}
  where $B(q,r)$ is a ball with radius $r>0$ and center at $q$, then $f$ is superharmonic.
\end{cor}
\begin{proof}
Consider the cut-off function as in Theorem $\ref{r9}$. Now using $(\ref{r7})$ and integration by parts, we get

\begin{eqnarray*}
0\leq \int_{B(q,2r)}\psi_rR = \int_{B(q,2r)} \psi_r (\lambda + \frac{1}{n} \Delta f) &= & \int_{B(q,2r)} \psi_r \lambda + \frac{1}{n} \int_{B(q,2r)} \psi_r \Delta f\\
&\leq& \int_{B(q,2r)}\psi_r \Delta \psi + \frac{1}{n} \int_{B(q,2r)}\psi_r \Delta f\\
&\leq& \int_{B(q,2r)-B(q,r)}\psi \Delta \psi_r + \frac{1}{n} \int_{B(q,2r)-B(q,r)} f \Delta \psi_r\\
&\leq& \int_{B(q,2r)-B(q,r)}\frac{C}{r^2} \big(\frac{f}{n}+\psi \big) \rightarrow 0,
\end{eqnarray*} 
as $r\rightarrow \infty$. Since $\psi_r =1$ in $B(q,r)$, so from above inequality we have $R=0$.
Then from equation (\ref{r12}), we get
\begin{equation*}
\Delta f \leq 0,
\end{equation*}
which proves the result.
\end{proof}
In \cite{ABR2011} Aquino et. al proved a result for gradient Ricci soliton. We have proved the same result for gradient Yamabe soliton as follows:
\begin{thm}\label{r21}
Let $(M,g,X,\lambda)$ be a gradient Yamabe soliton of dimension n which is compact, with potential function $f$. Then upto a constant, $f$ agrees with the Hodge-de Rham potential $h$.
\end{thm}
\begin{proof}
For Yamabe soliton $(M,g,X,\lambda)$, it follows that
\begin{equation}
(R-\lambda)n=divX.
\end{equation}
Now with the help of Hodge-de Rham decomposition we get $divX=\Delta h$, and hence
\begin{equation}\label{r18}
(R-\lambda)n=\Delta h.
\end{equation}
Again from $(\ref{r11})$ we have
 \begin{equation}\label{r19}
 \Delta f = (R-\lambda)n.
 \end{equation}
 Now, from equations $(\ref{r18})$ and $(\ref{r19})$, we get 
 \begin{equation}
 \Delta(f-h)=0.
 \end{equation}
 Therefore, we have $f=h+C$, for some constant C, which completes the proof of the theorem.
\end{proof}
\section{acknowledgment}
 The second author gratefully acknowledges to the
 CSIR(File No.:09/025(0282)/2019-EMR-I), Govt. of India for financial assistance.


\begin{thebibliography}{3}

 \bibitem{ABR2011}
 Aquino, C., Barros, A., Ribeiro, E. Jr., \textit{Some applications of the Hodge-de Rham decomposition to Ricci solitons}, Results. Math. \textbf{60} (2011), 245--254.
 
 \bibitem{GA2014}
  Calvaruso, G. and Zaeim, A., \textit{A complete classification of Ricci and Yamabe solitons of nonreductive homogeneous 4-spaces}, J. Geom. Phys., \textbf{80} (2014), 15--25.
 
 \bibitem{HXY2012}
  Cao, H. D., Sun, X. and  Xiaofeng, Y.S., \textit{ On the structure of gradient Yamabe solitons}, Math. Res. Lett., \textbf{19} (2012), 767--774. 
  
  \bibitem{CC1996}
 Cheeger, J. and Colding, T. H., \textit{Lower bounds on Ricci curvature and the almost rigidity of warped products}, Ann. Math., \textbf{144(1)} (1996), 189--237.
 

 \bibitem{PN2013}
 Daskalopoulos, P.,  Sesum, N., \textit{The classification of locally conformally flat Yamabe solitons}, Advances in Math., \textbf{240} (2013), 346--369.  
 
 
\bibitem{RS1988}
Hamilton, R.S., \textit{The Ricci flow on surfaces}, Mathematics and general relativity (Santa Cruz, CA, 1986), contemp. Math., \textbf{71} (1988), 237--262.
 \bibitem{SYH2012}
   Hsu, S. Y., \textit{A note on compact gradient Yamabe solitons}, J. Math. Anal. Appl., \textbf{388} (2012), 725-–726.
            

 \bibitem{LV2012}
   Ma, L. and  Miquel, V., \textit{ Remarks on scalar curvature of Yamabe solitons},  Ann. Global Anal.Geom., \textbf{42} (2012), 195--205. 
 \bibitem{SC10}
       Schoen, R., Yau, S.T., \textit{Lectures on differential geometry. Conference Proceedings
            and lecture notes in geometry and topology}, 
            I. International Press, Cambridge, MA, 1994.
            
  \bibitem{WF1983}
    Warner, F., \textit{Foundations of differentiable manifolds and Lie groups}, Springer- Verlag, New York, ISBN 978-0-387-90894-6, 1983. 
    
    \bibitem{ZHZ2009}
    Zhang, H. Z., \textit{On the completeness of gradient Ricci solitons}, Proc. Amer. Math. Soc., \textbf{137} (2009), 2755--2759.
    
    \bibitem{SM2019} 
    Shaikh, A. A. and Mondal, C. K., \textit{Non-existence of Riemannian metric satisfying Yamabe soliton}, 	arXiv:1908.00334.
             
            %%%%%%%%%%%%%%%%%%%%%%%%%%%%%%  
\end{thebibliography}
\end{document}